\documentclass{article}

\usepackage{amsmath}
\usepackage{amssymb}

\newtheorem{theorem}{Theorem}[section]

\newtheorem{lemma}[theorem]{Lemma}
\newtheorem{proposition}[theorem]{Proposition}

\newtheorem{remark}{Remark}[section]

\newtheorem{definition}{Definition}[section]

\newenvironment{proof}{\vspace{1ex}\noindent{\bf Proof:} }{
\mbox{}\hspace{\fill}\rule{1ex}{1ex}\vspace{1.5ex}}

\begin{document}

\title{Dispersion Models for Extremes}
\author{Bent J\o rgensen$^{1}$ $\cdot $ Yuri Goegebeur$^{1}$ $\cdot $ Jos%
\'{e} Ra\'{u}l Mart\'{\i}nez$^{2}$ }
\maketitle

\footnotetext[1]{%
Bent J\o rgensen (corresponding author), Yuri Goegebeur \newline
Department of Statistics, University of Southern Denmark, Campusvej 55,
DK-5230 Odense M, Denmark\newline
e-mail: bentj@stat.sdu.dk, yuri.goegebeur@stat.sdu.dk}\footnotetext[2]{%
Jos\'{e} Ra\'{u}l Mart\'{\i}nez\newline
FAMAF, Universidad Nacional de C\'{o}rdoba, Ciudad Universitaria, 5000 C\'{o}%
rdoba, Argentina\newline
e-mail: jmartine@mate.uncor.edu}

\begin{quotation}
\noindent \textbf{Abstract} We propose extreme value analogues of natural
exponential families and exponential dispersion models, and introduce the
slope function as an analogue of the variance function. The set of quadratic
and power slope functions characterize well-known families such as the
Rayleigh, Gumbel, power, Pareto, logistic, negative exponential, Weibull and
Fr\'{e}chet. We show a convergence theorem for slope functions, by which we
may express the classical extreme value convergence results in terms of
asymptotics for extreme dispersion models. The main idea is to explore the
parallels between location families and natural exponential families, and
between the convolution and minimum operations. \medskip \newline
\end{quotation}

\section{Introduction}

In a seminal paper Morris (1982) asked the following question: what do the
normal, Poisson, gamma, binomial, and negative binomial distributions have
in common that makes them so special? His answer was that they are all
natural exponential families with quadratic variance functions. This idea
has wide-ranging practical and theoretical ramifications, in particular for
generalized linear models (McCullagh and Nelder, 1989) and exponential
dispersion models (J\o rgensen, 1987).

Many subsequent authors have used the variance function as a
characterization and convergence tool for natural exponential families and
exponential dispersion models, cf.~J\o rgensen (1997), Casalis (2000) and
references therein. In particular, Tweedie (1984) and several authors
independently of him (Morris, 1981; Hougaard, 1986; Bar-Lev and Enis, 1986),
proposed and investigated the class of power variance functions,
corresponding to what we now call the Tweedie class of exponential
dispersion models. J\o rgensen et al.~(1994) showed that the Tweedie models
appear as limits in a class of convergence results for exponential
dispersion models, extending certain classical stable convergence results.

These ideas appear, at first sight, to have little relevance for extreme
value theory. Echoing Morris (1982) we may ask, however, what distributions
like the Rayleigh, Gumbel, power, Pareto, logistic and negative exponential
have in common that makes them so special in the context of extremes? Also,
is there an extreme value analogue of power variance functions, perhaps
related to the Weibull and Fr\'{e}chet distributions? In the present paper
we develop an extreme value dispersion model framework in the spirit of J\o %
rgensen (1997), leading to constructive answers to these questions. In
particular we find a close parallel between the above-mentioned Tweedie
convergence results and the classical extreme convergence results by Fisher
and Tippett (1928) and Gnedenko (1943). See Coles (2001), Kotz and Nadarajah
(2002) and Beirlant et al.~(2004) for background material on extremes.

In Section \ref{CandD} we introduce the rate and slope of a distribution as
analogues of the mean and variance, respectively. In Section \ref{HLfam}
hazard location families and slope functions are introduced as analogues of
natural exponential families and variance functions, respectively. In
Section \ref{secXD} we introduce extreme dispersion models as analogues of
exponential dispersion models. In Section \ref{secQHS} we classify quadratic
slope functions in a manner similar to Morris' (1982) classification of
quadratic variance functions. In Section \ref{secGEV} we draw a parallel
between generalized extreme value distributions and Tweedie models, having
power slope functions and power variance functions, respectively. In Section %
\ref{secConv} a general convergence result for slope functions is shown,
leading to a new proof of the classical extreme value convergence results,
now set in the extreme dispersion model setting. Finally, in Section \ref%
{SecExponential} we consider characterization and convergence for
exponential slope functions.

\section{Basic framework\label{CandD}}

We now introduce the basic setup for the paper, define the notions of rate
and slope for a real random variable $Y$, and show that they are analogues
of the mean and variance, respectively. It is convenient to use familiar
terms from lifetime analysis such as survival function and hazard function,
although $Y$ is not restricted to be positive. It is also convenient to use
minimum (min) rather than the conventional maximum. Results for maxima may
be obtained by a reflection in the usual way, see Beirlant et al.~(2004,
p.~46).

\subsection{Survival, hazard and integrated hazard}

We assume that the survival function $G(y)=P(Y\geq y)$ is twice continuously
differentiable on the support $\mathcal{C}=(a,b)\subseteq \mathbb{R}$,
continuous at $a$, and possibly discontinuous at $b$. Let $\mathcal{G}$
denote the set of all $G$ such that the density function $f=-G^{\prime }$ is
strictly positive on $\mathcal{C}$, and let $\mathcal{G}_{0}$ denote the
subset of $\mathcal{G}$ for which $0\in \mathcal{C}$.

When $G(b)>0$ we talk about \emph{right censoring} at $b$. In particular we
allow a positive probability mass $G(\infty )$ at $b=\infty $, in which case 
$G$ represents an \emph{improper} distribution. In survival analysis $%
G(\infty )$ is the probability that an individual never experiences the
event in question, see e.g. Aalen (1988).

We define the \emph{integrated hazard function} $H:\mathbb{R}\rightarrow
\lbrack 0,\infty ]$ and the \emph{hazard} \emph{function} $h:\mathbb{R}%
\rightarrow \lbrack 0,\infty ]$ corresponding to $G$ by $H(y)=-\log G(y)$
and $h(y)=H^{\prime }(y)$, respectively. It is understood that both $H$ and $%
h$ are $0$ to the left and $\infty $ to the right of $\mathcal{C}$, except
that $H(b)$ is finite if $G(b)>0$. With these conventions the following
relationship holds for all $y\in \mathbb{R}$, 
\begin{equation}
H(y)=\int_{-\infty }^{y}h(x)\,dx\text{.}  \label{H(y)}
\end{equation}

\subsection{Rate, slope and semiinvariants}

We now define the rate and slope for $Y$, and make the connection with the
min operation.

By way of motivation, let us recall the derivation of the mean and variance
from the moment generating function. Let the random variable $Y$ have moment
generating function $M(t)=\mathrm{E}\left( e^{tY}\right) $, with domain $%
\Theta =\left\{ t\in \mathbb{R}:M(t)<\infty \right\} $ such that $0\in 
\mathrm{int\,}\Theta $. Consider the \emph{cumulant generating function} $%
\kappa =\log M$ whose derivatives at zero $\kappa ^{(i)}(0)$ are the
cumulants. The mean and variance, in particular, are given by 
\begin{equation}
\mathrm{E}(Y)=\tau (0)\qquad \text{and}\qquad \mathrm{Var}(Y)=\tau ^{\prime
}(0)\text{,}  \label{evar}
\end{equation}%
where $\tau =\kappa ^{\prime }$ is the \emph{mean value mapping,} which is
strictly increasing on $\mathrm{int}\Theta $.

We now propose $G(y)=\mathrm{E}\left( 1_{Y\geq y}\right) $ as an analogue of
the moment generating function $M(t)=\mathrm{E}\left( e^{tY}\right) $, which
in turn makes $H$ and $h$ analogues of the cumulant generating function $%
\kappa $ and mean value mapping $\tau $, respectively. By analogy with (\ref%
{evar}), we define the \emph{rate}\textsl{\ }$r$ and \emph{slope} $s$ for a
random variable $Y$ with survival function $G\in \mathcal{G}_{0}$ by%
\begin{equation*}
r(Y)=h(0)\text{ and }s(Y)=h^{\prime }(0)\text{,}
\end{equation*}%
respectively. Unlike the variance, however, the slope may be negative as
well as positive, and the rate decreases (increases) under translation in
the IFR (DFR) case. In general we define the $i$th \emph{semiinvariant} by $%
k_{i}(Y)=H^{(i)}(0)$ for $i\geq 1$, provided the derivatives exist,
analogously to the cumulants. This terminology alludes to T.N. Thiele's name 
\emph{half-invariants} for the cumulants, cf.~Lauritzen (2002, p.~207).

Letting $\mu =r(Y)$, the slope of $Y$ may be written as follows: 
\begin{equation}
s(Y)=\mu \left\{ \mu -g^{\prime }\left( 0\right) \right\} \text{,}
\label{varform}
\end{equation}%
where $g=-\log f$. This result is somewhat analogous to the result $\mathrm{%
Var}(Y)=\mathrm{E}(Y^{2})-\mathrm{E}^{2}(Y)$ for the variance, see also
Section \ref{secExpo}.

Like the cumulants, the semiinvariants satisfy a scale equivariance property 
$k_{i}(cY)=c^{-i}k_{i}(Y)$ for $c>0$, which follows from the fact that $cY$
has integrated hazard function $H(y/c)$. In particular the rate and slope
satisfy 
\begin{equation}
r(cY)=c^{-1}r(Y)\text{ and }s(cY)=c^{-2}s(Y)\text{.}  \label{equivariance}
\end{equation}%
The rate is \emph{not}, however, translation equivariant, nor is the slope
translation invariant, but instead satisfy $r(c+Y)=h(-c)$ and $%
s(c+Y)=h^{\prime }(-c)$ for $c\in \mathbb{R}$.

The min of $n$ independent variables $Y_{i}$ has integrated hazard function $%
H_{1}(y)+\cdots +H_{n}(y)$, so the semiinvariants are additive with respect
to the min operation $\wedge $, in much the same way that the cumulants are
additive with respect to convolution. In particular%
\begin{equation}
r\left( \bigwedge_{i=1}^{n}Y_{i}\right) =\sum_{i=1}^{n}r(Y_{i})\text{ and }%
s\left( \bigwedge_{i=1}^{n}Y_{i}\right) =\sum_{i=1}^{n}s(Y_{i})\text{.}
\label{additivity}
\end{equation}%
We denote the scaled min of $n$ independent and identically distributed
(i.i.d.) variables $Y_{i}$ by 
\begin{equation}
\hat{Y}_{n}=n\bigwedge_{i=1}^{n}Y_{i}\text{,}  \label{heware}
\end{equation}%
which in many ways behaves like the sample mean. Thus, combining (\ref%
{equivariance}) and (\ref{additivity}) yields%
\begin{equation}
r\left( \hat{Y}_{n}\right) =\mu \text{ and }s\left( \hat{Y}_{n}\right) =%
\frac{\varsigma }{n}\text{,}  \label{mevare}
\end{equation}%
where $\mu $ is the rate and $\varsigma $ the slope of $Y_{i}$. Also, the
exponential distribution is invariant under the transformation (\ref{heware}%
), behaving like a constant does under averaging. This suggest a law of
large numbers involving the exponential distribution, as we shall now see.

\subsection{Exponential distribution\label{secExpo}}

Let $E_{\mu }$ denote an exponential variable with rate $\mu >0$. By a
shifted exponential variable we mean $a+E_{\mu }$ with $a<0$, whose support
includes $0$. For such a variable we find 
\begin{equation}
r\left( a+E_{\mu }\right) =\mu \text{ and }s\left( a+E_{\mu }\right) =0\text{%
,}  \label{sevare}
\end{equation}%
parallel to the form for the mean and variance of a constant. In the
notation of (\ref{varform}), note that $g^{\prime }(0)=\mu $ for the
variable $a+E_{\mu }$, so (\ref{varform}) implies that the slope is a signed
measure of the deviation of $Y$ from exponentiality, in much the same way
that the variance is a measure of the deviation of $Y$ from being constant.

Since the exponential distribution hence plays the role of constant in the
present setup, it is not surprising that there is a law of large numbers, as
suggested by (\ref{mevare}) and (\ref{sevare}), that involves convergence to
the exponential distribution. In fact, the scaled min $\hat{Y}_{n}$, after
left truncation at $0$, has survival function given, for $y>0$, by%
\begin{equation}
\frac{G^{n}\left( y/n\right) }{G^{n}\left( 0\right) }=\exp \left\{ -y\mu -%
\frac{\varsigma }{2n}y^{2}+o(n^{-1})\right\} \text{ as }n\rightarrow \infty 
\text{,}  \label{conditional}
\end{equation}%
converging to an exponential distribution with rate $\mu $. Here left
truncation at $0$ means conditioning on the event $\hat{Y}_{n}>0$. The
quadratic term suggests a central limit theorem. By removing the term $-y\mu 
$, corresponding to an exponential component, and rescaling we obtain for $%
y>0$ 
\begin{equation}
\frac{G^{n}\left( y/\sqrt{n}\right) }{G^{n}\left( 0\right) }e^{y\mu \sqrt{n}%
}=\exp \left\{ -\frac{\varsigma }{2}y^{2}+o(1)\right\} \text{ as }%
n\rightarrow \infty \text{.}  \label{CLT}
\end{equation}%
Provided $\varsigma >0$, this gives an asymptotic Rayleigh distribution,
which hence plays the role of the normal distribution in the present setup.
Remark \ref{minexp} below makes precise the idea of removing an exponential
component.

\section{Hazard location families\label{HLfam}}

\subsection{Motivation}

We now introduce hazard location families, and show that each such family is
characterized by its slope function, just like a natural exponential family
is characterized by its variance function.

Recall that the \emph{variance function} is defined on $\Omega =\tau (%
\mathrm{int\,}\Theta )$ by $V(\mu )=\tau ^{\prime }\left( \tau ^{-1}(\mu
)\right) $, where $\tau $ is the mean value mapping defined in connection
with (\ref{evar}). As pointed out by Morris (1982), $V$ characterizes is the
distribution of $Y$ up to an exponential tilting. This follows since given $%
V $, the inverse $\tau ^{-1}$ satisfies the differential equation 
\begin{equation}
\frac{d\tau ^{-1}(\mu )}{d\mu }=\frac{1}{V(\mu )}\text{,}  \label{tauv}
\end{equation}%
from which $\tau ^{-1}(\mu )$ may be recovered up to an additive constant $%
-\theta $, say, corresponding to an \emph{exponential tilting} of the
density $f$ (not necessarily with respect to Lebesgue measure) 
\begin{equation}
f(y;\theta )=f(y)\exp \left\{ y\theta -\kappa (\theta )\right\} \text{.}
\label{ne}
\end{equation}%
This is a natural exponential family (NEF), and (\ref{ne}) has mean $\mu
=\tau (\theta )$ and variance $V(\mu )$, cf.~J\o rgensen (1997, Ch.~2).

\subsection{Definition}

Our analogy implies that the hazard function $h$ should be analogous to the
mean value mapping $\tau $. In order to make the analogy complete, however,
we need $h$, like $\tau $, to be monotone. We thus consider from now on
survival functions in $\mathcal{G}$ with \emph{monotone hazard rate}, in the
sense that $h$ is strictly monotone on $\mathcal{C}$, either increasing
(IFR) or decreasing (DFR). This subset of $\mathcal{G}$ is denoted $%
\overline{\mathcal{G}}$, and we let $\overline{\mathcal{G}}_{0}=\overline{%
\mathcal{G}}\cap \mathcal{G}_{0}$.

\begin{remark}
\label{Remark1}In the DFR case it is necessary that $a>-\infty $ in order
for the integral (\ref{H(y)}) to converge at $a$. In the IFR case $G$ is
always proper, since if $h$ is increasing on $(a,\infty )$ then the integral
(\ref{H(y)}) diverges at $\infty $.
\end{remark}

\begin{remark}
\label{Remark2}Consider the case $Y=\log T$, where $T$ is a positive
survival time, say. Since then $\mathcal{C}=\mathbb{R}$, making $a=-\infty $%
, only IFR is possible. The derivative of the hazard function for $T$ is 
\begin{equation*}
\frac{h^{\prime }(\log t)-h(\log t)}{t^{2}}\text{.}
\end{equation*}%
Hence $T$ need not have monotone hazard rate, even though $Y$ does. In this
sense, the assumption of monotone hazard rate is less of a restriction when
modelling log survival times.
\end{remark}

We now define an analogue of the variance function. For given $G\in 
\overline{\mathcal{G}}$ we let $\Psi =h(\mathcal{C})$ (an open interval),
and define the\textbf{\ }\emph{slope function} $v:\Psi \rightarrow \mathbb{R}%
_{\pm }$ by 
\begin{equation}
v(\mu )=h^{\prime }\left( h^{-1}(\mu )\right) \text{.}  \label{three}
\end{equation}%
Here $v$ maps into $\mathbb{R}_{+}$ in the IFR case and into $\mathbb{R}_{-}$
in the DFR case. Analogously to (\ref{tauv}) we find that the inverse hazard
function $h^{-1}$ satisfies 
\begin{equation}
\frac{dh^{-1}(\mu )}{d\mu }=\frac{1}{v(\mu )}\text{.}  \label{hv}
\end{equation}

\begin{proposition}
\label{hazard}The slope function $v$ with domain $\Psi $ characterizes the
location family $G(\cdot -\theta )$ with $\theta \in \mathbb{R}$ among all
location families within $\overline{\mathcal{G}}$.
\end{proposition}

\begin{proof}
Given $v$, the solution to (\ref{hv}) is $\theta +h^{-1}(\mu )$, where $%
\theta \in \mathbb{R}$ is arbitrary. By inversion, we obtain the hazard
function $h(\cdot -\theta )$ corresponding to the location family $G(\cdot
-\theta )$.
\end{proof}

To complete the analogy with natural exponential families we restrict the
domain of $\theta $ to $-\mathcal{C}$, such that $G(\cdot -\theta )\in 
\overline{\mathcal{G}}_{0}$. Note that the rate and slope for $G(\cdot
-\theta )$ are $\mu =h\left( -\theta \right) $ and $h^{\prime }\left(
-\theta \right) =v(\mu )$. This leads to the following definition.

\begin{definition}
\label{defHL}The \emph{hazard location family} $\left\{ \mathrm{HL}(\mu
):\mu \in \Psi \right\} \subseteq \overline{\mathcal{G}}_{0}$ generated from 
$G\in \overline{\mathcal{G}}$ is defined by the family of survival functions
with support $\mathcal{C}-$ $h^{-1}(\mu )$ given by 
\begin{equation}
y\mapsto G\left\{ y+h^{-1}(\mu )\right\} \text{.}  \label{G(yh)}
\end{equation}
\end{definition}

Note that the definition of $v$ in (\ref{three}) is independent of the
representation (\ref{G(yh)}) used for the family, so the slope function
represents an intrinsic property of the family.

\begin{table}[tbp]
\caption{The main quadratic hazard slope families and associated NEFs. GHS
is the generalized hyperbolic secant family.}
\label{quadratic}\centering%
\begin{tabular}{|c|c|c|c|c|c|}
\hline
HL$(\mu )$ & $G(y)$ & $\mathcal{C}$ & $v(\mu )$ & $\Psi $ & NEF \\ \hline
Rayleigh & $\exp \left( -y^{2}/2\right) $ & $\mathbb{R}_{+}$ & $1$ & $%
\mathbb{R}_{+}$ & Normal \\ 
Gumbel & $\exp \left( -e^{y}\right) $ & $\mathbb{R}$ & $\mu $ & $\mathbb{R}%
_{+}$ & Poisson \\ 
Uniform & $1-y$ & $(0,1)$ & $\mu ^{2}$ & $(1,\infty )$ & Gamma \\ 
Pareto & $y^{-1}$ & $(1,\infty )$ & $-\mu ^{2}$ & $(0,1)$ & --- \\ 
Logistic & $\left( 1+e^{y}\right) ^{-1}$ & $\mathbb{R}$ & $\mu (1-\mu )$ & $%
(0,1)$ & Binomial \\ 
Neg. exponential & $1-e^{y}$ & $\mathbb{R}_{-}$ & $\mu (1+\mu )$ & $\mathbb{R%
}_{+}$ & Neg. binomial \\ 
Cosine & $\cos y$ & $(0,\pi /2)$ & $1+\mu ^{2}$ & $\mathbb{R}_{+}$ & GHS \\ 
\hline
\end{tabular}%
\end{table}

Table \ref{quadratic} shows some examples of hazard location families
corresponding to familiar distributions, all with quadratic slope functions
(polynomials of degree at most two), to be studied in Section \ref{secQHS}.
Except for the Pareto distribution, all the families in the table are IFR.
These six IFR families have the same functional form for $v$ as the variance
functions for Morris' (1982) six natural exponential families. In
particular, the Rayleigh distribution has constant slope function, like the
variance function of the normal distribution.

\subsection{Truncation and censoring\label{secTrans}}

We now study the effect on $v$ of transformations like truncation and
censoring.

Left truncation at some point $c\in \mathcal{C}$ gives rise to a new hazard
location model with $v$ restricted to a subset of $\Psi $. Similarly, the
operation of right censoring at some point $c\in \mathcal{C}$ corresponds to
replacing $Y$ by $Y\wedge c$, also known as the limited loss variable in
loss modelling (Klugman et al., 2004, p.~30). This operation reduces $%
\mathcal{C}$ to the subset $(a,c)$ and introduces the probability $G(c)$ in $%
c$. We summarize these considerations in a lemma.

\begin{lemma}
\label{Prop2}Left truncation at $c\in \mathcal{C}$ corresponds to
restricting the domain of $v$ to the interval between $h(c)$ and $h(b)$.
Right censoring at $c\in \mathcal{C}$ corresponds to restricting the domain
of $v$ to the interval between $h(a)$ and $h(c)$.
\end{lemma}

In the DFR case, left truncation thus results in the domain $\Psi
=(h(b),h(c))$, whereas right censoring gives the domain $(h(c),h(a))$. The
lemma shows that the restriction of $v$ to a subinterval of its domain is
again the slope function for a hazard location family. When looking for a
model corresponding to a given functional form for $v$, we may hence
concentrate on the largest possible domain consistent with a survival
function in $\overline{\mathcal{G}}_{0}$. Note, however, that restricting
the domain of $v$ to a subset of $\Psi $ implies a change in the
distributional form, because the support is changed. By comparison,
restricting $\mu $ to a subinterval of $\Omega $ in a natural exponential
family selects a subset of the family of distributions, without changing the
distributions as such.

\section{Extreme dispersion models\label{secXD}}

\subsection{Motivation}

We now introduce extreme dispersion models as a parallel to exponential
dispersion models, and show that they satisfy a reproductive property.

Given a natural exponential family (\ref{ne}) with variance function $V(\mu
) $, the corresponding \emph{exponential dispersion model} $\mathrm{ED}(\mu
,\lambda )$ consists of natural exponential families with variance function $%
\lambda ^{-1}V(\mu )$ proportional to $V(\mu )$. The latter is then called
the \emph{unit variance function}. The model $\mathrm{ED}(\mu ,\lambda )$
has density function of the form 
\begin{equation*}
f(y;\theta ,\lambda )=f_{\lambda }(y)\exp \left[ \lambda \left\{ y\theta
-\kappa (\theta )\right\} \right] \text{,}
\end{equation*}%
for a suitable function $f_{\lambda }$. Here $\mu =\tau (\theta )$ is the
mean, in the notation of (\ref{evar}), and $\sigma ^{2}=1/\lambda $ is the 
\emph{dispersion parameter}. The index parameter $\lambda $ has domain $%
\Lambda \subseteq \mathbb{R}_{+}$, which is an additive semigroup (often $%
\mathbb{R}_{+}$ or $\mathbb{N}$).

$\mathrm{ED}(\mu ,\lambda )$ satisfies the following \emph{mean reproductive
property}. The average of $n$ i.i.d. variables $Y_{1},\ldots ,Y_{n}$ from $%
\mathtt{\mathrm{ED}}(\mu ,\lambda )$ has distribution 
\begin{equation}
\bar{Y}_{n}\sim \mathtt{\mathrm{ED}}(\mu ,n\lambda )\text{,}  \label{yabar}
\end{equation}%
where the index parameter is proportional to the sample size. This follows
from the form of the moment generating function of $\mathrm{ED}(\mu ,\lambda
)$, which is%
\begin{equation}
t\mapsto \frac{M^{\lambda }(t/\lambda +\tau ^{-1}(\mu ))}{M^{\lambda }(\tau
^{-1}(\mu ))}\text{,}  \label{MGF}
\end{equation}%
where $M(\cdot )$ is the moment generating function for $f=f_{1}$, cf.~J\o %
rgensen (1997, Ch. 3).

\subsection{Definition}

\begin{definition}
Given a survival function $G\in \overline{\mathcal{G}}$ with hazard function 
$h$ and support $\mathcal{C}$, we define the \emph{extreme dispersion model
generated by }$G$\emph{,} denoted $\mathrm{XD}(\mu ,\lambda )$, as the
family of survival functions in $\overline{\mathcal{G}}_{0}$ given by%
\begin{equation}
y\mapsto G^{\lambda }(y/\lambda +h^{-1}(\mu ))  \label{XDdef}
\end{equation}%
with rate $\mu \in \Psi $, index parameter $\lambda >0$ and support $\lambda
\left( \mathcal{C}-h^{-1}(\mu )\right) $.
\end{definition}

It is straightforward to see that the model $\mathrm{XD}(\mu ,\lambda )$ may
be generated from any of its members in this way, up to a rescaling of $%
\lambda $ to $c\lambda $ for some $c>0$. In the following we work with the
representation (\ref{XDdef}) corresponding to a specific choice for $G$. The
corresponding hazard and density functions are for $y\in \lambda \left( 
\mathcal{C}-h^{-1}(\mu )\right) $ given by%
\begin{equation}
h(y;\mu ,\lambda )=h(y/\lambda +h^{-1}(\mu ))  \label{haz}
\end{equation}%
and 
\begin{equation*}
f(y;\mu ,\lambda )=h(y/\lambda +h^{-1}(\mu ))\exp \left[ -\lambda H\left\{
y/\lambda +h^{-1}(\mu )\right\} \right] \text{,}
\end{equation*}%
respectively. The right extreme of the support is $\lambda (b-h^{-1}(\mu ))$%
, which has probability $G^{\lambda }(b)$.

The parameter $\mu $ is the rate for $\mathrm{XD}(\mu ,\lambda )$ for any
value of $\lambda >0$, which follows from (\ref{haz}) by inserting $y=0$, in
much the same way that $\mu $ is the mean of $\mathrm{ED}(\mu ,\lambda )$
for all $\lambda $. For each fixed value of $\lambda $, (\ref{XDdef})
corresponds to a hazard location model with slope function $\lambda
^{-1}v(\mu )$. Hence $v$ is called the \emph{unit slope function} for $%
\mathtt{\mathrm{XD}}(\mu ,\lambda )$, and by Proposition \ref{hazard} $v$
characterizes $\mathtt{\mathrm{XD}}(\mu ,\lambda )$ up to a rescaling of $%
\lambda $ like above. We call $\sigma ^{2}=1/\lambda $ the \emph{dispersion
parameter}.

The following \emph{min reproductive property} easily follows from the form
of the survival function (\ref{XDdef}). For i.i.d. variables $Y_{1},\ldots
,Y_{n}\sim \mathtt{\mathrm{XD}}(\mu ,\lambda )$ the scaled min $\hat{Y}_{n}$
from (\ref{heware}) has distribution

\ 
\begin{equation}
\hat{Y}_{n}\sim \mathtt{\mathrm{XD}}(\mu ,n\lambda )\text{.}  \label{repro}
\end{equation}%
This is analogous to the mean reproductive property (\ref{yabar}) for
exponential dispersion models. It is equivalent to the max-stable property
of an exponentiated family of distributions (Nelson and Doganaksoy, 1995;
Sarabia and Castillo, 2005; Nadarajah and Kotz, 2006), but the present
formulation emphasizes the fact that the rate is preserved under the scaled
min operation. In (\ref{repro}), like (\ref{yabar}), the index parameter is
proportional to the sample size.

Let us consider the $\mathtt{\mathrm{XD}}(\mu ,\lambda )$ models generated
from first six cases in Table \ref{quadratic}, where in fact the
introduction of the index parameter $\lambda $ corresponds to known
generalizations. In the Rayleigh and Gumbel cases, this adds a scale or
location parameter to the models, respectively. The uniform distribution
becomes a shifted power distribution. The Pareto becomes a shifted
generalized Pareto distribution. The logistic becomes a generalized logistic
distribution. The negative exponential becomes the negative exponentiated
exponential, see Nadarajah and Kotz (2006).

We note in passing the well-known fact that a transformation of the variable 
$Y\sim \mathtt{\mathrm{XD}}(\mu ,\lambda )$ to the cumulated hazard scale $%
H(Y/\lambda +h^{-1}(\mu ))$ gives an exponential variable with parameter $%
\lambda $, possibly right censored at the point $H(b)$.

\subsection{Frailty models}

The study of frailty models reveals a certain intimate connection between
extreme and exponential dispersion models. Let the conditional distribution $%
Y|X=x$ be exponential with parameter $x$, and let $X$ be a non-negative
random variable with moment generating function $M(t)$. Then the marginal
survival function for $Y$ is $M(-y)$ for $y>0$, which is in effect the \emph{%
frailty model} of Vaupel et al. (1979).

In the special case where $X\sim \mathtt{\mathrm{ED}}(\mu ,\lambda )$ with
moment generating function (\ref{MGF}) we obtain the following survival
function for $Y$ (Hougaard, 1986):%
\begin{equation*}
G(y)=\frac{M^{\lambda }(-y/\lambda +\tau ^{-1}(\mu ))}{M^{\lambda }(\tau
^{-1}(\mu ))}\text{.}
\end{equation*}%
This is an extreme dispersion model with hazard function $h(y)=\tau (-y)$,
and corresponding unit slope function $v(\mu )=-V(\mu )$, the negative of
the unit variance function for $Y$. This model is hence DFR. An example is
the generalized Pareto distribution, which is the frailty model
corresponding to a gamma frailty, with slope function $-\mu ^{2}$ for $\mu
\in (0,1)$. As this example illustrates, the domain for $v$ may be a proper
subset of that for $V$.

In the particular case where $X$ has a positive probability at $0$, the
distribution of $Y$ becomes improper with $P(Y=\infty )=P(X=0)>0$. When $X$
follows the Tweedie compound Poisson distribution with $1<p<2$ (cf. J\o %
rgensen, 1997, Ch. 4), which has a positive probability at $0$, an improper
distribution for $Y$ is obtained, as pointed out by Aalen (1988).

\subsection{Exponential convergence}

We shall now return to the exponential convergence of Section \ref{secExpo}.
By way of motivation, note that an exponential dispersion variable $Y\sim 
\mathtt{\mathrm{ED}}(\mu ,\lambda )$ convergences in probability to $\mu $
as $\lambda \rightarrow \infty $, as is clear from (\ref{yabar}). The
analogous result for extreme dispersion models involves convergence to the
exponential distribution.

\begin{proposition}
\label{xdtoexe}For $Y\sim \mathtt{\mathrm{XD}}(\mu ,\lambda )$ and $c\in 
\mathbb{R}$ the conditional distribution of $Y-c$ given $Y>c$ is
asymptotically exponential with rate $\mu $ for $\lambda \rightarrow \infty $%
.
\end{proposition}

\begin{proof}
Let $\lambda $ be large enough to make the support $\lambda \left( \mathcal{C%
}-h^{-1}(\mu )\right) $ contain $c$. Then the conditional survival function
of $Y-c$ given $Y>c$ is, for $y>0$, 
\begin{equation*}
G_{c}(y;\mu ,\lambda )=\frac{G^{\lambda }(\left( c+y\right) /\lambda
+h^{-1}(\mu ))}{G^{\lambda }(c/\lambda +h^{-1}(\mu ))}\text{.}
\end{equation*}%
A Taylor expansion of $H$ around $c/\lambda +h^{-1}(\mu )$ gives 
\begin{equation*}
G_{c}(y;\mu ,\lambda )=\exp \left[ -yh\left\{ c/\lambda +h^{-1}(\mu
)\right\} -\frac{y^{2}}{2\lambda }h^{\prime }\left\{ c_{\lambda
\,y}+h^{-1}(\mu )\right\} \right] \text{,}
\end{equation*}%
where $c_{\lambda \,y}$ is between $c/\lambda $ and $\left( c+y\right)
/\lambda $. Letting $\lambda \rightarrow \infty $ and using the continuity
of $h$ and $h^{\prime }$, we obtain the desired result.
\end{proof}

\section{Quadratic slope functions\label{secQHS}}

We now follow Morris' (1982) footsteps and classify the set of quadratic
slope functions. To this end, we need to study reflections of slope
functions, and the role of exponential components. These transformations
have a somewhat formal nature, but turn out to be useful for the
classification result. For the sake of brevity, certain details in this
section are left to the reader.

\subsection{Reflections}

We now consider what happens when we subject $v$ to a horizontal or vertical
reflection.

\begin{proposition}
\label{reflection}Let $G\in \overline{\mathcal{G}}_{0}$ have support $%
\mathcal{C}=(a,b)$ and slope function $v$. \emph{Horizontal reflection: }If $%
G$ is right censored, then the survival function $y\mapsto G(b)/G(-y)$ with
support $-\mathcal{C}$ has hazard function $h(-y)$ and slope function $-v$
on $\Psi $, and is also right censored. \emph{Vertical reflection:} Assume
that $(0,m)\subseteq \Psi $ and restrict the support to the interval $%
(a_{0},b_{0})$, either $(a,h^{-1}(m))$ (IFR case) or $(h^{-1}(m),b)$ (DFR
case). If $G$ is right censored at $b_{0}<\infty $ then the survival
function $y\mapsto G(-y)/G(b_{0})\exp \left\{ -m(y+b_{0})\right\} $ with
support $(-b_{0},-a_{0})$ has slope function $\mu \mapsto v(m-\mu )$ with
domain $(0,m)$, and is right censored if $a_{0}>-\infty $.
\end{proposition}

\begin{proof}
Horizontal reflection: The survival function $G(b)/G(-y)$ with support $-%
\mathcal{C}$ is easily seen to have hazard function $h(-y)$ and slope
function $-v(\mu )$ on $\Psi $. The value at the right endpoint is $%
G(b)/G(a)>0$, so the model is right censored. Vertical reflection: The
survival function $G(-y)/G(b_{0})\exp \left\{ -m(y+b_{0})\right\} $ with
support $(-b_{0},-a_{0})$ is similarly seen to have hazard function $m-h(-y)$
and slope function $v(m-\mu )$ on $(0,m)$. If $a_{0}>-\infty $ then $%
G(a_{0})/G(b_{0})\exp \left\{ -m(b_{0}-a_{0})\right\} >0$, so in this case
the model is right censored.
\end{proof}

Table \ref{reflections} shows three hazard location families with quadratic
slope functions obtained by vertical reflection of families from Table \ref%
{quadratic}. 
\begin{table}[tb]
\caption{Some vertically reflected quadratic hazard slopes. }
\label{reflections}\centering%
\begin{tabular}{|c|c|c|c|c|}
\hline
HL$(\mu )$ & $G(y)$ & $\mathcal{C}$ & $v(\mu )$ & $\Psi $ \\ \hline
Reflected Gumbel & $\exp \left( 1-y-e^{-y}\right) $ & $\mathbb{R}_{+}$ & $%
1-\mu $ & $(0,1)$ \\ 
Reflected Logistic & $1/\cosh 
{\frac12}%
y$ & $\mathbb{R}_{+}$ & $\left( 
{\frac12}%
-\mu \right) (%
{\frac12}%
+\mu )$ & $(0,%
{\frac12}%
)$ \\ 
Reflected neg. exponential & $4\left( e^{-y}-e^{-2y}\right) $ & $(\log
2,\infty )$ & $\left( \mu -1\right) (\mu -2)$ & $(0,1)$ \\ \hline
\end{tabular}%
\end{table}

\subsection{Exponential components\label{secExco}}

Extending the results of Section \ref{secExpo}, we now show that an
exponential component (in the sense of Remark \ref{minexp} below)
corresponds to a location change for the slope function.

\begin{proposition}
\label{excompo}Let $G\in \overline{\mathcal{G}}_{0}$ with support $\mathcal{C%
}=(a,b)$ have slope function $v$ with domain $\Psi =(\underline{\eta },%
\overline{\eta })$. If $a>-\infty $ then for $m\geq -\underline{\eta }$, the
function $v(\mu -m)$ with domain $m+\Psi $ is the slope function of the
survival function given by $G(y)\exp \left\{ -m(y-a)\right\} $ on $\mathcal{C%
}$.
\end{proposition}

\begin{proof}
The survival function $G(y)\exp \left\{ -m(y-a)\right\} $ has integrated
hazard function 
\begin{equation}
m\left( y-a\right) +H(y)  \label{explus}
\end{equation}%
on $\mathcal{C}$, provided $m\geq -\underline{\eta }$, with hazard function $%
m+h(y)$ and slope function $v(\mu -m)$ on $m+\Psi $.
\end{proof}

\begin{remark}
\label{minexp}A positive $m$ in (\ref{explus}) corresponds to the variable $%
\min \left\{ Y,a+E_{m}\right\} $, where the exponential variable $E_{m}$ is
independent of $Y$. In this case we say that we are \emph{introducing} an
exponential component. Conversely, when $m=-\underline{\eta }<0$ we say that
we are \emph{removing} the exponential component, making $\inf \Psi =0$. The
only model in Table \ref{quadratic} with an exponential component is the
uniform distribution. After removing the exponential component, we obtain 
\begin{equation}
G(y)=e^{y}(1-y)\text{ for }y\in (0,1)\text{.}  \label{expaway}
\end{equation}%
The corresponding slope function is $\left( 1+\mu \right) ^{2}$ with domain $%
\Psi =\mathbb{R}_{+}$.
\end{remark}

Note that, using the terminology of Remark \ref{minexp}, it is understood in
connection with the classification results below that vertical reflection
(Proposition \ref{reflection}) is applied only after removing the
exponential component, if necessary, to ensure that $\inf \Psi =0$.

An example of an exponential component is encountered in connection with the
Gumbel family with unit slope function $v(\mu )=\mu $ on $\Psi =\mathbb{R}%
_{+}$. The Gompertz-Makeham distribution is obtained from the Gumbel by left
truncation at $0$, restricting $v$ to $\mu >1$, and then adding an
exponential component. This gives the hazard function $h(y)=m+e^{\beta y}$
for $m,\beta ,y>0$ and unit slope function $v(\mu )=\mu -m$ for $\mu >1+m$.
A horizontal reflection, corresponding to $\beta <0$, yields $v(\mu )=m-\mu $
for $\mu \in (m,m+1)$.

\subsection{Classification}

We have now considered several types of transformations of slope functions,
including censoring, truncation and reflections. In addition to these, we
consider the following three transformations of a given hazard location
family $\mathrm{HL}(\mu )$ with slope function $v$ and domain $\Psi $.

\begin{enumerate}
\item Location change: removing or adding an exponential component maps $v$
into $v(\mu -m)$.

\item Scale transformation: a scale transformation of $Y$ maps $v$ into $%
c^{-2}v(c\mu )$ for $c>0$.

\item Multiplication: generating an extreme dispersion model maps $v$ into $%
v/\lambda $ for $\lambda >0$.
\end{enumerate}

A combination of these three operations maps $v$ into%
\begin{equation}
\gamma v\left( \left( \mu -\alpha \right) /\beta \right) \text{,}
\label{scaling}
\end{equation}%
where $\gamma ,\beta >0$ and $\mu \in \alpha +\beta \Psi $. We refer to (\ref%
{scaling}) as the operation of \emph{location and scaling}. This leads us to
the main classification theorem.

\begin{theorem}
\label{qsf}Up to left truncation, right censoring, reflection, location and
scaling, the only hazard location models with quadratic slope functions are
those shown in Table \ref{quadratic}.
\end{theorem}

The next remark will useful for the proof.

\begin{remark}
\label{remSubstitute}Consider $G\in \overline{\mathcal{G}}$ and $c,d\in 
\mathcal{C}$. The following identity 
\begin{equation}
H(d)-H(c)=\int_{h(c)}^{h(d)}\frac{\mu }{v(\mu )}\,d\mu  \label{substitute}
\end{equation}%
follows by the substitution $\mu =h(x)$ in the integral (\ref{H(y)}). It is
useful for checking if a given function $v$ may serve as a slope function.
By taking $c=a$, we find that the continuity of $G$ at $a$ is equivalent to
the integral (\ref{substitute}) being convergent at $h(a)$. By taking $d=b$,
we find that right censoring is equivalent to the integral (\ref{substitute}%
) being convergent at $h(b)$.
\end{remark}

\begin{proof}
[of Theorem \ref{qsf}] By means of the location and scaling operation we may
reduce the classification problem to quadratic slope functions with simple
forms like in Table \ref{quadratic}, having roots either $\pm 1$, $0$ or $i$%
. A combination of vertical and horizontal reflections applied to the seven
cases of Table \ref{quadratic} then covers all possible shapes of quadratic
slope functions, most of which are right censored (cf. Proposition \ref%
{reflection}). Regarding the uniform and Pareto distributions, an
application of Remark \ref{remSubstitute} shows that neither $\mu ^{2}$ nor $%
-\mu ^{2}$ can be slope functions on $\Psi =\mathbb{R}_{+}$, but only on a
subset of $\mathbb{R}_{+}$. It follows that horizontal reflections of the
uniform and Pareto distributions give rise to two separate cases of right
censored slope functions of the form $\pm \mu ^{2}$, and a further four
cases of the form $(\mu -m)^{2}$ that follow by vertical reflection. It is
easily seen that this covers all possible cases.
\end{proof}

\begin{remark}
\label{remCensor}There are three cases with $\inf \mathcal{C}=-\infty $ in
Table \ref{quadratic} where a vertical reflection leads to models that are
not right censored, of which typical examples are shown in Table \ref%
{reflections}.
\end{remark}

One could also explore parallels of other classification results for NEFs,
such as Letac and Mora's (1990) cubic variance functions, but this is
outside the scope of the present paper.

\section{Generalized extreme value distributions\label{secGEV}}

We now investigate the analogy between the generalized extreme value
distribution and the Tweedie class of exponential dispersion models. The
latter is characterized by having unit variance functions of power form $%
V(\mu )=\mu ^{p}$, cf.~J\o rgensen (1997, Ch. 4) and references therein.
Here $p=0$ corresponds to the normal distribution with domain $\mathbb{R}$,
whereas the remaining cases, namely $p<0$ and $p\geq 1$, all have domain $%
\mathbb{R}_{+}$. The special cases $p=0,1,2$ all appear in Table~\ref%
{quadratic}, and a further simple case is $p=3$, corresponding to the
inverse Gaussian distribution.

\subsection{Definition}

The standard generalized extreme value distribution for minima is defined by 
\begin{equation*}
G(y)=\exp \left\{ -\left( 1-\gamma y\right) ^{-1/\gamma }\right\} \text{,}
\end{equation*}%
with support defined by $\gamma y<1$. Here $\gamma \in \mathbb{R}$, and the
value $\gamma =0$ (defined by continuity) corresponds the Gumbel
distribution. All extreme value distributions except the exponential ($%
\gamma =-1$) have monotone hazard rates and their slope functions are of
power form 
\begin{equation}
v(\mu )=\frac{1}{2-p}\mu ^{p}  \label{power}
\end{equation}%
for $\mu >0$, where the parameter $p\in \mathbb{R\setminus }\left\{
2\right\} $ is defined by 
\begin{equation}
p=p(\gamma )=\frac{1+2\gamma }{1+\gamma }\text{.}  \label{p(ga)}
\end{equation}%
The models are IFR for $p<2$ ($\gamma >-1$) and DFR for $p>2$ ($\gamma <-1$%
). As we saw in the proof of Theorem \ref{qsf} there is no slope function on 
$\mathbb{R}_{+}$ proportional to $\mu ^{2}$. Table \ref{extreme} summarizes
the main cases of generalized extreme value distributions corresponding to
different values of $p$. 
\begin{table}[tbp]
\caption{Summary of generalized extreme value distributions.}
\label{extreme}\centering%
\begin{tabular}{|c|c|c|c|}
\hline
EV$_{\gamma }(\mu ,\lambda )$ & $\gamma $ & $p$ & Support \\ \hline
Weibull & $\gamma <-1$ & $p>2$ & $\left( 1/\gamma ,\infty \right) $ \\ 
Exponential & $\gamma =-1$ & $-$ & $\left( -1,\infty \right) $ \\ 
Weibull & $-1<\gamma <0$ & $p<1$ & $\left( 1/\gamma ,\infty \right) $ \\ 
Gumbel & $\gamma =0$ & $p=1$ & $\mathbb{R}$ \\ 
Fr\'{e}chet & $\gamma >0$ & $1<p<2$ & $\left( -\infty ,1/\gamma \right) $ \\ 
\hline
\end{tabular}%
\end{table}

Introducing location and index parameters, the generalized extreme value
distributions are seen to be examples of extreme dispersion models, one for
each $\gamma $. We thus define the $\mathrm{EV}_{\gamma }(\mu ,\lambda )$ to
be the extreme dispersion model given by the survival function%
\begin{equation*}
y\mapsto \exp \left\{ -\lambda \left( \mu ^{-\gamma /(1+\gamma )}-\gamma
y/\lambda \right) ^{-1/\gamma }\right\} \text{,}
\end{equation*}%
with support defined by $\gamma y<\lambda \mu ^{-\gamma /(1+\gamma )}$,
which is a reparametrization of the usual extreme value distribution. The $%
\mathrm{EV}_{\gamma }(\mu ,\lambda )$ model satisfies the following scaling
property 
\begin{equation}
c\mathrm{EV}_{\gamma }(\mu ,\lambda )=\mathrm{EV}_{\gamma }(c^{-1}\mu
,c^{2-p}\lambda )\text{.}  \label{cEV}
\end{equation}

\subsection{Characterization}

The Tweedie models may be characterized as the only exponential dispersion
models closed under scale transformations, cf.~J\o rgensen (1997, p.~128).
We now show, by means of Proposition \ref{hazard} that the extreme value
distributions satisfy a similar property.

\begin{theorem}
\label{XDscale}Let \textrm{$XD$}$(\mu ,\lambda )$ be such that for some $%
\lambda >0$ and all $\mu ,c>0$%
\begin{equation}
c\mathrm{XD}(\mu ,\lambda )=\mathrm{XD}(c^{-1}\mu ,g_{\lambda }(c))
\label{scalexd}
\end{equation}%
for some positive function $g_{\lambda }(c)$. Then \textrm{$XD$}$(\mu
,\lambda )$ is a generalized extreme value distribution.
\end{theorem}

\begin{proof}
First note that the rate $c^{-1}\mu $ on the right-hand side of (\ref%
{scalexd}) is consistent with (\ref{equivariance}). Since $c>0$ is arbitrary
this in turn implies that $\Psi =\mathbb{R}_{+}$. Without loss of generality
we may take $\lambda =1$. Calculating the slope function on both sides of (%
\ref{scalexd}) gives%
\begin{equation*}
c^{-2}v(\mu )=\frac{1}{g_{1}(c)}v(c^{-1}\mu )\text{.}
\end{equation*}%
This implies that $v$ satisfies the functional equation $v(x)v(y)=v(1)v(xy)$
for $x,y>0$. Using the continuity of $v$, the solution is $v(\mu )=c_{p}\mu
^{p}$, where $p\in \mathbb{R}$ and $c_{p}$ is an arbitrary non-zero constant
that may depend on $p$. When $p\neq 2$ and $c_{p}=1/(2-p)$ this
characterizes the generalized extreme value distribution $\mathrm{EV}%
_{\gamma }(\mu ,\lambda )$. Other choices for $c_{p}$ with the same sign
correspond to a scale change. Changing the sign to $c_{p}=-1/(2-p)$ is
possible only for a right censored survival function (Proposition \ref%
{reflection}), which is incompatible with the condition $\Psi =\mathbb{R}%
_{+} $. The case $p=2$, which has been dealt with in Section \ref{secQHS},
also is not compatible with the condition $\Psi =\mathbb{R}_{+}$. It easily
follows that $g_{\lambda }(c)=\lambda c^{2-p}$, in agreement with (\ref{cEV}%
).
\end{proof}

\section{Convergence of extremes\label{secConv}}

\subsection{General convergence theorem}

The results of the previous section show that the Tweedie and generalized
extreme value distributions share certain properties due to the common form
of their variance and slope functions. We shall now complete this analogy by
showing a convergence theorem for slope functions, which in turn leads to a
new proof of the extreme value convergence theorem, along the same lines as
the Tweedie convergence theorem of J\o rgensen et al.~(1994).

The use of variance functions for proving convergence for natural
exponential families was initiated by Morris (1982), but a rigorous
formulation and proof was first given by Mora (1990). The convergence
theorem for variance functions says that if a sequence of variance functions
converges uniformly on compact sets, then the corresponding sequence of
natural exponential families converges to the family corresponding to the
limiting variance function. We have the following analogous result for slope
functions. The proof is given in an Appendix.

\begin{theorem}
\label{vntov}Let $v_{n}$, $\Psi _{n}=(\underline{\eta }_{n},\overline{\eta }%
_{n})$ be a sequence of slope functions and their respective domains, all
IFR (DFR), such that $\Psi =\mathrm{int}\left( \lim_{n\rightarrow \infty
}\Psi _{n}\right) $ exists and is non-empty, where $\lim_{n\rightarrow
\infty }\Psi _{n}$ means that each of the two sequences of endpoints
converges. Assume that $v_{n}$ converges on $\Psi $, uniformly on compact
subintervals of $\Psi $, to a function $v$ which is strictly positive
(strictly negative) on $\Psi $. Assume that $v_{n}$ satisfies the following
left tightness condition. For each $k>0$ there exists an $\eta \in \Psi $
such that for all $n$%
\begin{equation}
\int_{I_{n}(\eta )}\frac{\mu }{\left\vert v_{n}(\mu )\right\vert }\,d\mu <k%
\text{,}  \label{tight}
\end{equation}%
where $I_{n}(\eta )=(\underline{\eta }_{n},\eta )$ ($I_{n}(\eta )=(\eta ,%
\overline{\eta }_{n})$). Then the corresponding sequence of hazard location
families $\mathrm{HL}_{n}(\mu )$ converges weakly for each $\mu \in \Psi $,
uniformly on compact subintervals of the support, to the hazard location
family $\mathrm{HL}(\mu )$ with slope function $v$.
\end{theorem}

The tightness condition (\ref{tight}) originates from the identity (\ref%
{substitute}). The following remark shows that a similar condition is useful
for determining if the limiting family is right censored.

\begin{remark}
Under the assumptions of Theorem \ref{vntov}, we consider the following
right tightness condition. For each $k>0$ there exists an $\eta \in \Psi $
such that for all $n$%
\begin{equation*}
\int_{I_{n}(\eta )}\frac{\mu }{\left\vert v_{n}(\mu )\right\vert }\,d\mu >k%
\text{,}
\end{equation*}%
where $I_{n}(\eta )=(\eta ,\overline{\eta }_{n})$ (IFR case) or $I_{n}(\eta
)=(\underline{\eta }_{n},\eta )$ (DFR case). Then for every $k>0$ there
exists a $c\in \mathcal{C}$ such that $H_{n}(c)>k$ for all $n$. This, in
turn, implies that $H(c)=\lim H_{n}(c)>k$, and hence $h(b-)=\infty $. This
implies no right censoring, so in particular the limiting distribution is
proper.
\end{remark}

\subsection{Extreme convergence theorem}

Let $\gamma \neq -1$ be given, and let $p=p(\gamma )\neq 2$, according to (%
\ref{p(ga)}). Choosing $c$ in the scaling formula (\ref{cEV}) such that $%
n=c^{2-p}$ is an integer, we obtain%
\begin{equation}
n^{1/(p-2)}\mathrm{EV}_{\gamma }(n^{1/(p-2)}\mu ,n\lambda )=\mathrm{EV}%
_{\gamma }(\mu ,\lambda )\text{.}  \label{(*)}
\end{equation}%
Now recall the min reproductive property (\ref{repro}), by which the
left-hand side of (\ref{(*)}) represents a centering and scaling of the
scaled min $\hat{Y}_{n}$ for a sample of size $n$ from $\mathrm{EV}_{\gamma
}(\mu ,\lambda )$. In effect (\ref{(*)}) represents the so-called \emph{%
stability postulate} for the limiting distribution of extremes, cf. Kotz and
Nadarajah (2002, p.~5). The corresponding domains of attraction correspond
to the classical extreme convergence result, which in the present setup
takes the following form.

\begin{theorem}
\label{exconv}Let $\mathtt{\mathrm{XD}}(\mu ,\lambda )$ be an extreme
dispersion model having unit slope function $v$ with power asymptotics of
the form 
\begin{equation}
v(\mu )\sim \frac{1}{2-p}\mu ^{p}  \label{poweras}
\end{equation}%
as $\mu \rightarrow 0$ (IFR case with $p<2$) or $\mu \rightarrow \infty $
(DFR case with $p>2$). Then for any $\mu ,\lambda >0$ 
\begin{equation}
n^{1/(p-2)}\mathtt{\mathrm{XD}}(n^{1/(p-2)}\mu ,n\lambda )\overset{w}{%
\longrightarrow }\mathrm{EV}_{\gamma }(\mu ,\lambda )\text{ as }n\rightarrow
\infty \text{,}  \label{EVconv}
\end{equation}%
where $\overset{w}{\longrightarrow }$ denotes weak convergence.
\end{theorem}

\begin{proof}
Consider the IFR case $p<2$, where the power asymptotics holds near $0$. For
fixed values of $\lambda $ and $n$, the left-hand side of (\ref{EVconv}) is
a hazard location family with slope function 
\begin{equation*}
v_{n}(\mu )=\frac{1}{\lambda n^{p/(p-2)}}v(n^{1/(p-2)}\mu )\longrightarrow 
\frac{1}{\lambda \left( 2-p\right) }\mu ^{p}\text{ as }n\rightarrow \infty 
\text{,}
\end{equation*}%
where we have used the scaling property of the slope in (\ref{equivariance}%
). The pointwise convergence follows from (\ref{poweras}). To show that the
convergence is uniform in $\mu $ on compact subsets of $\mathbb{R}_{+}$, let 
$0<\mu <m$ for given $m>0$. For given $\varepsilon >0$ let $\mu _{0}$ be
such that 
\begin{equation*}
\left\vert \frac{v(\mu )}{\mu ^{p}}-\frac{1}{2-p}\right\vert <\varepsilon
\end{equation*}%
for $\mu <\mu _{0}$, by the assumption of power asymptotics. Then for any $n$
large enough to make $n^{1/(p-2)}<\mu _{0}/m$ we find 
\begin{equation*}
\left\vert \frac{v(n^{1/(p-2)}\mu )}{n^{p/(p-2)}}-\frac{\mu ^{p}}{2-p}%
\right\vert =\mu ^{p}\left\vert \frac{v(n^{1/(p-2)}\mu )}{\left( \mu
n^{1/(p-2)}\right) ^{p}}-\frac{1}{2-p}\right\vert \leq m^{p}\varepsilon
\end{equation*}%
for all $\mu <m$, which shows the uniform convergence. Since we are in the
IFR case, the tightness condition (\ref{tight}) involves the integral 
\begin{equation*}
\int_{0}^{\eta }\frac{\lambda n^{p/(p-2)}\mu }{v(n^{1/(p-2)}\mu )}\,d\mu 
\text{.}
\end{equation*}%
Since the integrand behaves asymptotically like the power $\mu ^{1-p}$,
which is integrable on $(0,\eta )$ for $p<2$, the tightness condition is
satisfied. The convergence (\ref{EVconv}) hence follows from Theorem \ref%
{vntov}. The proof in the DFR case $p>2$ is similar.
\end{proof}

Compared with conventional extreme value results, the framework of Theorem %
\ref{exconv} is very convenient, albeit under the rather strong conditions
of differentiability of the density, and monotone hazard rate. We note that
the condition (\ref{poweras}) seamlessly integrates the Gumbel case ($p=1$)
with the rest, whereas the exponential case is not included, see Remark \ref%
{expoconv}.

We note that the approach leads to the new centering constant $\theta
=-h^{-1}(n^{1/(p-2)}\mu )$, the location parameter appearing in (\ref{XDdef}%
), which corresponds to keeping the rate constant at the value $\mu $
throughout the convergence (\ref{EVconv}).

In simple cases, like in Table \ref{quadratic}, it is very easy to read off
the asymptotic behaviour of the slope function. For example, the asymptotic
behaviour of $v$ near $0$ for the logistic and negative exponential
distributions is $v(\mu )\sim \mu $, so both are in the domain of attraction
of the Gumbel distribution. Further examples are considered below.

\begin{remark}
The von Mises conditions are sufficient conditions involving the density $f$
for extreme value convergence. For $G\in \overline{\mathcal{G}}$ with
support $(a,b)$, the version of the Gumbel condition proposed by Falk and
Marohn (1993) is (keeping in mind that we use min rather than max)%
\begin{equation*}
\lim_{y\downarrow a}\frac{f(y)}{1-G(y)}=c
\end{equation*}%
for some $c>0$. By l'Hospital's rule this is equivalent to 
\begin{equation}
\lim_{y\downarrow a}\frac{h^{\prime }(y)}{h(y)}=c\text{.}  \label{hprime}
\end{equation}%
By inserting $y=h^{-1}(\mu )$, we find that (\ref{hprime}) is equivalent to (%
\ref{poweras}) with $p=1$. The situation for $p\neq 1$ is, however, less
clear. For $a=0$ the Gumbel condition is 
\begin{equation*}
\lim_{y\downarrow 0}\frac{yf(y)}{1-G(y)}=-\gamma ^{-1}>0\text{,}
\end{equation*}%
or equivalently, with an application of l'Hospital's rule, 
\begin{equation*}
\lim_{y\downarrow 0}\frac{yh^{\prime }(y)}{h(y)}=-1-\gamma ^{-1}\text{.}
\end{equation*}%
This condition apparently cannot be expressed conveniently in terms of the
slope function $v$.
\end{remark}

\begin{remark}
\label{expoconv}Contrary to conventional extreme value convergence theory,
our framework separates out the case of exponential convergence, and
Proposition \ref{xdtoexe} illustrates how exponential convergence is
prompted by left truncation, see also (\ref{conditional}). The uniform and
Pareto examples from Table \ref{quadratic} illustrate that distributions in
the domain of attraction of the exponential distribution have incomplete $%
\Psi $, with $\inf \Psi >0$ (IFR case) or $\sup \Psi <\infty $ (DFR case).
In the case of the uniform distribution with the exponential component
removed (\ref{expaway}), the new slope function $(1+\mu )^{2}$ satisfies (%
\ref{poweras}) with $p=0$, and so is in the domain of attraction of the
Rayleigh distribution.
\end{remark}

J\o rgensen and Mart\'{\i}nez (1997) developed Tauberian methods for
variance functions, where power asymptotics for $V$ is replaced by regular
variation. This could be developed in the present setting along the lines of
de Haan (1970), but is outside the scope of the present paper.

\subsection{Examples}

Let us consider two further examples of extreme value convergence that
illustrate Theorem \ref{exconv}. First we consider the \emph{negative Pareto
distribution}\textbf{\ }with survival function $G(y)=1-\left( 1-y\right)
^{-1}$ for $y<0$. Straightforward calculations show that the corresponding
slope function is%
\begin{equation}
v(\mu )=\mu \sqrt{\mu ^{2}+4\mu }\text{ for }\mu >0\text{,}  \label{Letac}
\end{equation}%
which behaves like $2\mu ^{3/2}$ near $0$. Letting $\mathtt{\mathrm{XD}}(\mu
,\lambda )$ denote the extreme dispersion model corresponding to $G$, an
application of Theorem \ref{exconv} yields Fr\'{e}chet convergence,%
\begin{equation*}
n^{-2}\mathtt{\mathrm{XD}}(n^{-2}\mu ,n\lambda )\overset{w}{\longrightarrow }%
\mathrm{EV}_{1}(\mu ,\lambda )\text{ as }n\rightarrow \infty \text{.}
\end{equation*}%
It is worth noting that $v$ in (\ref{Letac}) is of the so-called \emph{Letac
form} (J\o rgensen, 1997, pp.~157--158), a class of variance functions that
has been extensively studied, see e.g. Kokonendji (1994).

Next, we consider the \emph{Burr distribution }with survival function\textsl{%
\ }$G(y)=\left( 1+y^{\alpha }\right) ^{-1}$ for $y>0$, for some $\alpha >0$,
which is DFR for $0<\alpha \leq 1$. An explicit expression for the slope
function may be found in the case $\alpha =1/2$, where%
\begin{equation*}
v(\mu )=-\mu ^{2}\left( \mu +2+\sqrt{\mu ^{2}+2\mu }\right) \text{ for }\mu
>0.
\end{equation*}%
The asymptotic behaviour is $v(\mu )\sim -2\mu ^{3}$ as $\mu \rightarrow
\infty $. An application of Theorem \ref{exconv} yields Weibull convergence
with $\gamma =-2$,%
\begin{equation*}
n\mathtt{\mathrm{XD}}(n\mu ,n\lambda )\overset{w}{\longrightarrow }\mathrm{EV%
}_{-2}(\mu ,\lambda /2)\text{ as }n\rightarrow \infty \text{,}
\end{equation*}%
where $\mathtt{\mathrm{XD}}(\mu ,\lambda )$ denotes the extreme dispersion
model generated by $G$. For $0<\alpha <1$ the behaviour of $v$ is like $-\mu
^{p}$ with $p=(\alpha -2)/(\alpha -1)>2$, and (\ref{EVconv}) applies.

For $\alpha >1$ the Burr hazard is not monotone, but is for $y$ near $0$.
Hence by a suitable right censoring, we obtain an IFR model with asymptotic
behaviour $\mu ^{p}$ for $v$ with $p<1$. In general Theorem \ref{exconv} may
be applied in this way to models with non-monotone hazard as long as the
hazard is monotone near $0$.

\section{Exponential slope functions\label{SecExponential}}

We now consider characterization and convergence for exponential slope
functions, similar to J\o rgensen's (1997, p.~160) characterization of
exponential variance functions. These results have independent interest,
since exponential variance functions correspond to natural exponential
families generated by extreme stable distributions with stability index $%
\alpha =1$.

\subsection{Characterization}

Elaborating on the parallel between the Rayleigh and normal distributions,
we note that the latter satisfies the following transformation property: 
\begin{equation*}
\mathrm{N}(m+\mu ,\sigma ^{2})-m=\mathrm{N}(\mu ,\sigma ^{2})
\end{equation*}%
for all $m\in \mathbb{R}$, imitating (\ref{(*)}), but with multiplication
replaced by addition. From (\ref{CLT}), we would expect in the Rayleigh IFR
case that the term $-m$ corresponds to a left truncation followed by the
removal of an exponential component. More generally, given $\mathrm{XD}(\mu
,\lambda )$ with unit slope function $v$ on $\psi =\mathbb{R}_{+}$, we
consider the \emph{shift transformation}, defined by the following two steps.

\begin{enumerate}
\item Left truncation (IFR) or right censoring (DFR), which restricts the
domain to $\mu >m$, while maintaining the slope at $\lambda ^{-1}v(\mu )$.
This gives rise to an exponential component.

\item Removing the exponential component, giving the rate $\mu >0$ and slope 
$\lambda ^{-1}v(m+\mu )$.
\end{enumerate}

The result is an extreme dispersion model $\mathrm{XD}_{m}(\mu ,\lambda )$
with unit slope function $v(m+\cdot )$. We now characterize exponential
slope functions as fixed points for the shift transformation.

\begin{theorem}
\label{exponential}Let \textrm{$XD$}$(\mu ,\lambda )$ have unit slope
function $v$ and domain $\Psi =\mathbb{R}_{+}$. If for some $\lambda >0$
there exists a positive function $g_{\lambda }(m)$ such that for all $m,\mu
>0$ 
\begin{equation}
\mathrm{XD}_{m}(\mu ,g_{\lambda }(m))=\mathrm{XD}(\mu ,\lambda )\text{,}
\label{XDmuplusc}
\end{equation}%
then the unit slope function $v$ is either constant or exponential.
\end{theorem}

\begin{proof}
By calculating the slope on both sides of (\ref{XDmuplusc}) we obtain the
equation%
\begin{equation*}
\frac{1}{g_{\lambda }(m)}v(m+\mu )=\lambda ^{-1}v(\mu )\text{.}
\end{equation*}%
Without loss of generality we may take $\lambda =1$. By letting $\mu
\downarrow 0$ and using the continuity of $v$, we find that the limit $v(0+)$
exists, is positive and finite, and $g_{1}(m)=v(m)/v(0+)$. This, in turn,
implies that $v$ satisfies the functional equation $v(0+)v(m+\mu )=v(m)v(\mu
)$ for all $m,\mu >0$. Taking into account the continuity of $v$, the
solution is 
\begin{equation}
v(\mu )=v(0+)e^{\beta \mu }  \label{vexpo}
\end{equation}%
for some $\beta \in \mathbb{R}$, which in turn implies that (\ref{XDmuplusc}%
) holds for all $\lambda >0$ with $g_{\lambda }(m)=\lambda e^{\beta m}$.
\end{proof}

Besides the Rayleigh case ($\beta =0$) there are two main cases of (\ref%
{vexpo}), one IFR and one DFR. The IFR case has unit slope function $v(\mu
)=e^{-\mu }$ for $\mu >0$, and corresponds to the extreme dispersion model
generated from the survival function 
\begin{equation}
G(y)=e^{y}\left( 1+y\right) ^{-\left( 1+y\right) }\text{ for }y>0\text{.}
\label{IFR}
\end{equation}%
The DFR case has unit slope function $v(\mu )=-e^{\mu }$ for $\mu >0$, and
corresponds to the extreme dispersion model generated from the survival
function 
\begin{equation}
G(y)=e^{-y}y^{y}\text{ for }0<y<1\text{,}  \label{DFR}
\end{equation}%
which is right censored at $1$.

Note that by applying a suitable location and scaling operation to the power
slope function (\ref{power}) for $p>2$ we obtain%
\begin{equation*}
-\left( 1+\frac{\mu }{p}\right) ^{p}\rightarrow -e^{\mu }\text{ for }%
p\rightarrow \infty \text{,}
\end{equation*}%
which shows that the DFR case of (\ref{vexpo}) is a limiting case of the
generalized extreme value family. A similar result holds in the IFR case.

\subsection{Convergence\label{secExposlope}}

We now show a convergence theorem for exponential slope functions, similar
to a result for exponential variance functions (J\o rgensen, 1997, p.~164).
In effect, the fixed point (\ref{XDmuplusc}) has a domain of attraction
consisting of models with asymptotically exponential slope functions.

\begin{theorem}
Let $\mathrm{XD}(\mu ,\lambda )$ denote an extreme dispersion model with
unit slope function $v$ and domain $\Psi =\mathbb{R}_{+}$ having exponential
asymptotics of the form%
\begin{equation}
v(\mu )\sim c_{\beta }e^{\beta \mu }  \label{exponentiala}
\end{equation}%
for $\mu \rightarrow \infty $, where $c_{\beta }=1$ for $\beta \leq 0$ and $%
c_{\beta }=-1$ for $\beta >0$. Then the shifted model \textrm{$XD$}$_{m}(\mu
,\lambda e^{\beta m})$ converges to an extreme dispersion model with
exponential slope function for $m\rightarrow \infty $.
\end{theorem}

\begin{proof}
The shifted model $\mathrm{XD}_{m}(\mu ,\lambda e^{\beta m})$ has unit slope
function 
\begin{equation}
e^{-\beta m}v(m+\mu )\rightarrow c_{\beta }e^{\beta \mu }\text{ for }%
m\rightarrow \infty \text{,}  \label{exslope}
\end{equation}%
pointwise for $\mu >0$. To show that the convergence is uniform in $\mu $ on
compact subsets of $\mathbb{R}_{+}$, let $0<\mu <m_{0}$ for given $m_{0}>0$.
For given $\varepsilon >0$ let $\mu _{0}$ be such that 
\begin{equation*}
\left\vert e^{-\beta \mu }v(\mu )-c_{\beta }\right\vert <\varepsilon
\end{equation*}%
for $\mu >\mu _{0}$, by (\ref{exponentiala}). Then for any $m>\mu _{0}$ we
find 
\begin{equation*}
\left\vert e^{-\beta m}v(m+\mu )-c_{\beta }e^{\beta \mu }\right\vert
=e^{\beta \mu }\left\vert e^{-\beta \left( m+\mu \right) }v(m+\mu )-c_{\beta
}\right\vert \leq \left( 1+e^{\beta m}\right) \varepsilon
\end{equation*}%
for all $\mu <m_{0}$, showing uniform convergence. The tightness condition (%
\ref{tight}) involves the integral%
\begin{equation*}
\int_{I_{m}(\eta )}\frac{\lambda \mu }{\left\vert e^{-\beta m}v(m+\mu
)\right\vert }\,d\mu \text{,}
\end{equation*}%
where the integrand behaves asymptotically like $\mu e^{-\beta \mu }$. For $%
\beta >0$ the interval of integration is $(\eta ,\infty )$, whereas for $%
\beta \leq 0$ it is $(0,\eta )$, so in both cases the tightness condition is
satisfied. The result now follows from Theorem \ref{vntov}.
\end{proof}

There are three main cases of (\ref{exslope}). \emph{DFR case}. Take $\beta
=1$ and let $m$ be such that $n=e^{m}$ is an integer. We may then write the
convergence as follows:\ 
\begin{equation}
\mathrm{XD}_{\log n}(\mu ,\lambda n)\overset{w}{\rightarrow }\mathrm{XD_{-}}%
(\mu ,\lambda )\text{ for }n\rightarrow \infty \text{,}  \label{XDlog}
\end{equation}%
where $\mathrm{XD_{-}}(\mu ,\lambda )$ is the model generated by (\ref{DFR}%
). The left-hand side of (\ref{XDlog}) represents a shift transformation of
the scaled min $\hat{Y}_{n}$ for a sample of size $n$ from $\mathrm{XD}(\mu
,\lambda )$. \emph{Rayleigh case.} For $\beta =0$ we obtain convergence to
the Rayleigh distribution, 
\begin{equation*}
\mathrm{XD}_{m}(\mu ,\lambda )\overset{w}{\rightarrow }\mathrm{EV}_{-%
{\frac12}%
}(\mu ,\lambda )\text{ for }m\rightarrow \infty \text{.}
\end{equation*}%
\emph{IFR case.} Take $\beta =-1$ and let $t=e^{-m}$. Then 
\begin{equation*}
\mathrm{XD}_{-\log t}(\mu ,\lambda t)\overset{w}{\rightarrow }\mathrm{XD_{+}}%
(\mu ,\lambda )\text{ for }t\downarrow 0\text{,}
\end{equation*}%
where $\mathrm{XD_{+}}(\mu ,\lambda )$ is the model generated by (\ref{IFR}%
). This in effect involves the asymptotic distribution of an extremal
process $X_{t}$ for $t\downarrow 0$, much like the infinitely divisible type
of convergence of J\o rgensen (1997, p.~149). This follows by noting that to
every $\mathrm{XD}(\mu ,\lambda )$ model there exists an extremal process $%
X_{t}$, in the sense of Dwass (1964), such that $tX_{t}\sim \mathrm{XD}(\mu
,\lambda t)$.

\section*{Appendix: Proof of general convergence theorem\label{secApp}}

The proof of Theorem \ref{vntov} proceeds along the same lines as J\o %
rgensen's (1997, p.~54) proof of the convergence theorem for variance
functions, which in turn is a simplification of Mora's (1990) proof in the
multivariate case. The idea is to reconstruct the hazard function $h$ from
the limiting slope function $v$ using (\ref{hv}), and in turn use the
uniform convergence and tightness to show convergence of the sequence $H_{n}$%
.

Let $K$ be a given compact subinterval of $\Psi $. By assumption $\Psi =%
\mathrm{int}\left( \lim \Psi _{n}\right) $, so we may assume that $%
K\subseteq \Psi _{n}$ from some $n_{0}$ on. We only need to consider $%
n>n_{0} $. Fix a $\mu _{0}\in \mathrm{int}\,K$. Let $\psi _{n}=h_{n}^{-1}$
denote the inverse hazard function given by $\psi _{n}^{\prime }\left( \mu
\right) =1/v_{n}(\mu )$ on $\Psi _{n}$ and $\psi _{n}\left( \mu _{0}\right)
=0$, cf.~(\ref{hv}). Let $h_{n}$, $H_{n}$ etc. denote the quantities
associated with this parametrization.

Similarly, define $\psi :\Psi \rightarrow \mathbb{R}$ by $\psi ^{\prime
}\left( \mu \right) =1/v(\mu )$ on $\Psi $ and $\psi (\mu _{0})=0$. Then for 
$\mu \in K$ 
\begin{equation}
\left\vert \psi _{n}^{\prime }\left( \mu \right) -\psi ^{\prime }\left( \mu
\right) \right\vert =\frac{\left\vert v_{n}(\mu )-v(\mu )\right\vert }{%
v_{n}(\mu )v(\mu )}\text{.}  \label{psin}
\end{equation}%
By the uniform convergence of $v_{n}(\mu )$ to $v(\mu )$ on $K$, it follows
that $v_{n}(\mu )$ is uniformly bounded on $K$. Since $v(\mu )$ is bounded
on $K$, it follows from (\ref{psin}) and from the uniform convergence of $%
v_{n}$ that $\psi _{n}^{\prime }\left( \mu \right) \rightarrow \psi ^{\prime
}\left( \mu \right) $ uniformly on $K$. This and the fact that $\psi
_{n}\left( \mu _{0}\right) =\psi (\mu _{0})$ for all $n$ implies, by a
result from Rudin (1976, Theorem 7.17, p.~152), that $\psi _{n}\left( \mu
\right) \rightarrow \psi \left( \mu \right) $ uniformly on $K$.

Let $\mathcal{C}_{n}=\psi _{n}\left( \Psi _{n}\right) $ and $\mathcal{C}%
=\psi (\Psi )$. Then $\mathcal{C}=\mathrm{int}\left( \lim \mathcal{C}%
_{n}\right) $. Let $J=\psi (K)\subseteq \mathcal{C}$ and $J_{n}=\psi
_{n}(K)\subseteq \mathcal{C}_{n}$. Define $h:\mathcal{C}\rightarrow \Psi $
by $h(y)=\psi ^{-1}(y)$. Since $\psi $ is strictly monotone and
differentiable, the same is the case for $h$, and $h(y)>0$ on $\mathcal{C}$
since $\Psi \subseteq \mathbb{R}_{+}$. Let $\mu \in K$ be given and let $%
y=\psi (\mu )\in J$ and $y_{n}=\psi _{n}(\mu )\in J_{n}$. Since $v_{n}(\mu )$
is uniformly bounded on $K$, there exists an $m$ such that $\left\vert
v_{n}(\mu )\right\vert \leq m$ for all $n$ and $\mu \in K$. It follows that $%
\left\vert h_{n}^{\prime }(y)\right\vert \leq m$ for all $y\in J$. Since $%
\mu =h(y)=h_{n}(y_{n})$ we find, using the mean value theorem, that 
\begin{eqnarray*}
\left\vert h_{n}(y)-h(y)\right\vert &=&\left\vert
h_{n}(y)-h_{n}(y_{n})\right\vert \\
\ &\leq &m\left\vert y-y_{n}\right\vert \\
\ &=&m\left\vert \psi (\mu )-\psi _{n}(\mu )\right\vert \text{.}
\end{eqnarray*}%
This implies that $h_{n}(y)\rightarrow h(y)$ uniformly in $y\in J$. The
above arguments also apply if $J$ is extended to a larger subinterval of $%
\mathcal{C}$.

In order to invoke the tightness condition (\ref{tight}), we first consider
the IFR case. Using (\ref{substitute}) we obtain for $c\in \mathcal{C}$ 
\begin{equation}
H_{n}(c)=\int_{\underline{\eta }_{n}}^{h_{n}(c)}\frac{\mu }{\left\vert
v_{n}(\mu )\right\vert }\,d\mu \text{.}  \label{hnc}
\end{equation}%
For a given $\eta \in \Psi $ we may choose $\varepsilon >0$ and $c\in 
\mathcal{C}$ such that $h(c)+\varepsilon <\eta $, and from the convergence
of $h_{n}(c)$ to $h(c)$ we obtain $\underline{\eta }_{n}<h_{n}(c)<\eta $ for 
$n$ large enough. Together with (\ref{tight}) this implies that for every $%
k>0$ there exists a $c=c(k)\in \mathcal{C}$ such that for all $n$ 
\begin{equation}
0\leq H_{n}(c)\leq k\text{.}  \label{hight}
\end{equation}%
Since all $H_{n}$ are increasing we can make $c(k)$ an increasing function
of $k$. In the DFR case the inequality (\ref{hight}) follows similarly by
integrating over the interval $(h_{n}(c),\overline{\eta }_{n})$ in (\ref{hnc}%
).

The condition (\ref{hight}) implies that there exists a $c\in \mathcal{C}$
such that 
\begin{equation}
\lim \inf_{n\rightarrow \infty }\int_{-\infty }^{c}h_{n}(x)\,dx=\lim
\inf_{n\rightarrow \infty }H_{n}(c)<\infty \text{.}  \label{liminf}
\end{equation}%
By Fatou's lemma, (\ref{liminf}) implies that $\int_{-\infty
}^{c}h(x)\,dx<\infty $. We may now define $G$ and $H$ for $y\in \mathbb{R}$
by $H\left( y\right) =\int_{-\infty }^{y}h(x)\,dx$, and $G(y)=\exp \left\{
-H\left( y\right) \right\} $, where we use the conventions discussed in
connection with (\ref{H(y)}). Then $G$ is a survival function with support $%
\mathcal{C}$, and $H(\inf \mathcal{C})=0$. Using the above-mentioned result
from Rudin once more, we find that for any given $d$ in $J$, $H_{n}\left(
y\right) -H_{n}(d)$ converges to $H\left( y\right) -H(d)$ uniformly in $y\in
J$.

To conclude the proof, we choose a $d\in J$, and show that $H_{n}(d)$
converges to $H(d)$. The tightness condition (\ref{hight}) implies that, for
given $k>0$ and $c\leq c(k)$, $H_{n}(d)$ satisfies 
\begin{equation*}
H_{n}(d)-H_{n}\left( c\right) \leq H_{n}(d)\leq H_{n}(d)-H_{n}\left(
c\right) +k\text{.}
\end{equation*}%
We may enlarge $J$ to include $c$. Letting $n\rightarrow \infty $ we find
that $H_{n}(d)$ is asymptotically squeezed between the values $H(d)-H\left(
c\right) $ and $H(d)-H\left( c\right) +k$, which can be made arbitrarily
close to $H(d)$ by choosing $k$ small, and $c$ close to $\inf \mathcal{C}$.
Hence $H_{n}(d)$ converges to $H(d)$. It follows that $H_{n}\left( y\right) $
converges to $H\left( y\right) $ uniformly in $y\in J$, completing the proof.

\section*{Acknowledgements}

The research was supported by the Danish Natural Science Research Council.

\section*{References}

\begin{itemize}
\item Aalen, O.O., Heterogeneity in survival analysis, Statist. Med. \textbf{%
7}, 1121--1137, (1988).

\item Bar-Lev, S.K. and Enis, P., Reproducibility and natural exponential
families with power variance functions, Ann. Statist. \textbf{14},
1507--1522, (1986).

\item Beirlant, J., Goegebeur, Y., Segers, J. and Teugels, J., Statistics of
Extremes: Theory and Applications, Chichester: Wiley, 2004.

\item Casalis, M., Natural exponential families. In Kotz, S., Balakrishnan,
N., and Johnson, N.L. (ed.), Continuous Multivariate Distributions, Vol. 1:
Models and Applications, John Wiley \& Sons, New York, pp.~659--696, 2000.

\item Coles, S.G., An Introduction to Statistical Modeling of Extreme
Values, Springer-Verlag, London, 2001.

\item de Haan, L., On Regular Variation and Its Application to the Weak
Convergence of Sample Extremes, Mathematical Centre Tracts No. 32,
Mathematisch Centrum, Amsterdam, 1970.

\item Dwass, M., Extremal processes. Ann. Math. Statist. \textbf{35},
1718--1725, (1964).

\item Falk, M. and Marohn, F., von Mises conditions revisited. Ann. Probab. 
\textbf{21}, 1310--1328, (1993).

\item Fisher, R.A. and Tippett, L.H.C., Limiting forms of the frequency
distribution of the largest or smallest member of a sample. Proc. Cambridge
Phil. Soc. \textbf{24}, 180--190, (1928).

\item Gnedenko, B.V., Sur la distribution limite du terme maximum d'une s%
\'{e}rie al\'{e}atorire. Ann. Math. \textbf{44}, 423--453, (1943).

\item Hougaard, P., Survival models for heterogeneous populations derived
from stable distributions, Biometrika \textbf{73}, 387--396, (1986).

\item J\o rgensen, B., Exponential dispersion models (with discussion), J.
Roy. Statist. Soc. Ser. B \textbf{49}, 127--162, (1987).

\item J\o rgensen, B., The Theory of Dispersion Models, Chapman \& Hall,
London, 1997.

\item J\o rgensen, B., and Mart\'{\i}nez, J.R., Tauber theory for infinitely
divisible variance functions, Bernoulli \textbf{3}, 213--224, (1997).

\item J\o rgensen, B., Mart\'{\i}nez, J.R. and Tsao, M. Asymptotic behaviour
of the variance function, Scand. J. Statist. \textbf{21}, 223--243, (1994).

\item Klugman, S.A., Panjer, H.H. and Willmot, G.E., Loss Models: From Data
to Decisions, 2nd ed., Wiley Interscience, Hoboken, N.J., 2004.

\item Kotz, S. and Nadarajah, S., Extreme Value Distributions: Theory and
Applications, Imperial College Press, London, 2002.

\item Kokonendji, C.C., Exponential families with variance functions in $%
\sqrt{\Delta }P(\sqrt{\Delta })$: Seshadri's class, Test \textbf{3},
123--172, (1994).

\item Lauritzen, S.L., Thiele: Pioneer in Statistics, Oxford University
Press, Oxford, 2002.

\item Letac, G. and Mora, M., Natural real exponential families with cubic
variance functions, Ann. Statist. \textbf{18}, 1--37, (1990).

\item McCullagh, P. and Nelder, J.A., Generalized Linear Models, 2nd Ed.,
Chapman and Hall, London, 1989.

\item Mora, M., Convergence of the variance functions of natural exponential
families, Ann. Fac. Sci. Univ. Toulouse, S\'{e}rie 5 \textbf{11}, 105--120,
(1990).

\item Morris, C.N., Models for positive data with good convolution
properties, Memo no. 8949. California: Rand Corporation, (1981).

\item Morris, C.N., Natural exponential families with quadratic variance
functions, Ann. Statist. \textbf{10}, 65--80, (1982).

\item Nadarajah, S. and Kotz, S., The exponentiated type distributions, Acta
Appl. Math. \textbf{92}, 97--111, (2006).

\item Nelson, W. and Doganaksoy, N., Statistical analysis of life or
strength data from specimens of various sizes using the power-(log) normal
model. In Recent Advances in Life Testing and Reliability (ed. Balakrishnan,
N.), pp.~377--408, 1995.

\item Rudin, W., Principles of Mathematical Analysis, McGraw-Hill, New York,
1976.

\item Sarabia, J.M. and Castillo, E., About a class of max-stable families
with applications to income distributions, Metron LXIII, 505--527, (2005).

\item Tweedie, M.C.K., An index which distinguishes between some important
exponential families, In Ghosh, J.K. and Roy, J. (ed.), Statistics:
Applications and New Directions. Proceedings of the Indian Statistical
Institute Golden Jubilee International Conference, Indian Statistical
Institute, Calcutta, pp.~579--604, 1984.

\item Vaupel, J.W., Manton, K.G. and Stallard, E. The impact of
heterogeneity in individual frailty on the dynamics of mortality. Demography
16, 439--54, (1979). 
\end{itemize}

\end{document}